\documentclass[12pt]{amsart}

\usepackage{epsfig}
\usepackage{graphics}
\usepackage{mathrsfs}
\usepackage{epstopdf}
\usepackage{graphicx}
\usepackage{floatrow}
\usepackage[colorlinks]{hyperref}
\usepackage{amssymb}
\usepackage{verbatim}

\newtheorem{theorem}{Theorem}[section]

\newtheorem{lemma}[theorem]{Lemma}

\theoremstyle{definition}
\newtheorem{definition}[theorem]{Definition}
\newtheorem{example}[theorem]{Example}

\newtheorem{remark}[theorem]{Remark}

\theoremstyle{remark}

 \def\R{{\mathbb{R}}}
 \def\Z{{\mathbb{Z}}}

 \def\S{{\Sigma}}
 \def\C{{\mathbb{C}}}
 
 \def\sp{{\rm Sp}}
\def\h_1{{\rm H_1}}

\def\S{{\mathbb{S}}}
\def\D{{\mathbb{D}}}

\begin{document}

\newenvironment{prooff}{\medskip \par \noindent {\it Proof}\ }{\hfill
$\square$ \medskip \par}
    \def\sqr#1#2{{\vcenter{\hrule height.#2pt
        \hbox{\vrule width.#2pt height#1pt \kern#1pt
            \vrule width.#2pt}\hrule height.#2pt}}}
    \def\square{\mathchoice\sqr67\sqr67\sqr{2.1}6\sqr{1.5}6}
\def\pf#1{\medskip \par \noindent {\it #1.}\ }
\def\endpf{\hfill $\square$ \medskip \par}
\def\demo#1{\medskip \par \noindent {\it #1.}\ }
\def\enddemo{\medskip \par}
\def\qed{~\hfill$\square$}

 \title[ Signatures of Lefschetz Fibrations]
{Partial Fiber Sum Decompositions and Signatures of Lefschetz Fibrations}

\author{Adalet \c{C}engel}
 \address{Department of Mathematics, Bo\u{g}azi\c{c}i University, Bebek 34342}
 \address{School of Mathematics, University of Minnesota, Minneapolis,MN, 55455}
\email{{cenge002@umn.edu}}

\author{\c{C}a\u{g}ri Karakurt}
 \address{Department of Mathematics, Bo\u{g}azi\c{c}i University, Bebek 34342}
\email{{cagri.karakurt@boun.edu.tr}}

 \date{\today}

\begin{abstract}
In his Ph.D. thesis, Burak Ozbagci described an algorithm for computing signatures of Lefschetz fibrations where the input is a factorization of the monodromy into a product of Dehn twists. In this note, we give a reformulation of Ozbagci's algorithm which becomes much easier to implement. Our main tool  is Wall's non-additivity formula applied to what we call partial fiber sum decomposition of a Lefschetz fibration over the $2$-disk. We show that our algorithm works for bordered Lefschetz fibrations over disk and it yields a formula for the signature of branched covers where the branched loci are regular fibers. As an application, we give the explicit monodromy factorization of a Lefschetz fibration over disk whose total space has arbitrarily large positive signature for any positive fiber genus.    
\end{abstract}

 \maketitle

  \setcounter{secnumdepth}{2}
 \setcounter{section}{0}

\section{Introduction}

Donaldson's ground breaking result says that the study of symplectic $4$-manifolds up to blow-up is equivalent to that of Lefschetz fibrations over the $2$-sphere.  This relationship has been extended to the relative case by Akbulut-Ozbagci \cite{MR1825664} and Loi-Piergallini  \cite{MR1835390} who established a correspondence  between Stein manifolds and positive allowable Lefschetz fibrations. Working with Lefschetz  fibrations is more preferable as the topology of the total space, which is a  $4$-dimensional manifold, is completely determined by a monodromy factorization in mapping class groups, which can be understood through $2$-dimensional techniques. It remains a good challenge for $4$-dimensional topologist to compute invariants of a Lefschetz fibration out of its monodromy factorization. Signature is perhaps the simplest non-trivial  invariant among them. Recall that the signature $\sigma(Y)$ of a compact oriented $4$-manifold $Y$ is the signature of the intersection form on the second homology group $H_2(Y;\Z).$

Several computation techniques and formulas for signature of Lefschetz fibrations on $4$-manifolds exist in the literature. For Lefschetz fibrations over $\S^2$ with fiber genus $g=1,2$ Matsumoto \cite{Matsumoto1}; for hyperelliptic Lefschetz fibrations over surfaces with fiber genus $g \geq 3,$ Endo \cite{Endo} gave signature formulas. On the other hand, Endo and Nagami \cite{en} showed that the signature of a Lefschetz fibration over $\S^2$ can be calculated by using the signatures of relations contained in its monodromy. Ozbagci gave an algorithm which works for Lefschetz fibrations over $\D^2$ or $\S^2$ with closed fibers \cite{ozbagci}. Recently, Miyamura \cite{Miyamura} presented a formula for  Lefschetz fibrations over $\D^2$ with planar fibers. In this paper, we present an algorithm for computing the signature of Lefschetz fibrations over $\D^2$ of any genus $g$ fibers which can be closed or bordered.

A Lefschetz fibration on a smooth $4$-manifold $Y$ has a handlebody decomposition determined by a sequence of vanishing cycles \cite{Kas}. Ozbagci used this description and Wall non-additivity formula in his algorithm \cite{ozbagci}. The singular fibers of a Lefschetz fibration are obtained by attaching $2$-handles along the correspondig vanishing cycles in different regular fibers $\Sigma_g \times \D^2.$ For each attachment there is a signature contribution in the set $ \{-1,0,+1\}$ and their sum  gives the signature of the total space. Computing Ozbagci's local signature contribution is not a straightforward matter, as it is necessary to understand a presentation of the homology of the complement of the vanishing cycle in the boundary at each step.  The purpose of the present paper is to give an alternative method for computing these local signature using elementary linear algebra only. 

To state the main result, we fix our terminology first. Let $\Sigma_{g}^b$ be a compact, connected, oriented surface of genus $g$ with $b$ boundary components and ${\rm Mod}(\Sigma_{g}^b)$ denote its mapping class group, the group consisting of isotopy classes of orientation-preserving self-diffeomorphisms of $\Sigma_{g}^b$ which restrict to the identity on $\partial \Sigma_{g}^b.$ We write $\Sigma_{g}=\Sigma_{g}^0$ and ${\rm Mod}(\Sigma_{g})={\rm Mod}(\Sigma_{g}^0)$ for simplicity.

It is well-known that when the surface has no boundary components, the group ${\rm Mod}(\Sigma_{g})$ naturally  acts on $\h_1 (\Sigma_{g};\R)$ preserving the algebraic intersection form $Q$. Hence we have a canonical homomorphism from  ${\rm Mod}(\Sigma_{g})$ to  $\sp(2g, \R)$ called the symplectic representation of ${\rm Mod}(\Sigma_{g}).$  When $b\geq 1$ we pick a point  $p_j$ from each boundary component and  denote the set $\{p_1,\dots,p_b\}$ of distinguished points by $\mathfrak{d}$. The relative version of the symplectic representation for ${\mathrm{Mod}}(\Sigma_{g}^b)$ by   its action on $\h_1 (\Sigma_{g}^b;\mathfrak{d};\R)$ yields a homomorphism  ${\rm Mod}(\Sigma_{g}^b)\rightarrow {\sp}(2(g+b-1),\R)$.
It is an easy exercise to see that the above representation is obtained by connecting a distinguished boundary component  of $\Sigma_{g}^b$ to the others by adding one handles reducing the total number of boundary components to one and then capping off with a disk to obtain a closed surface of genus $g+b-1$  and seeing ${\rm Mod}(\Sigma_{g}^b)$ as a subgroup of ${\rm Mod}(\Sigma_{g+b-1})$.

Let $\gamma$ be a simple closed curve on $\Sigma_g^b$. We denote by $t_{\gamma}$, and $t_\gamma^*$ the Dehn twist about $\gamma$ and respectively the image of the Dehn twist $t_{\gamma}$ under this symplectic representation. The sign of a  real number $x$ is denoted by $\rm{sign}(x)\in \{-1,0,1\}$.   Our main result is the following.

\begin{theorem}\label{theo:main}
Given a Lefschetz fibration $f_{n}:Y_{n}\rightarrow \D^2$ with regular fiber $\Sigma_{g}^b$ and monodromy factorization $\phi=t_{\gamma_n} \cdots t_{\gamma_1};$ the signature $\sigma(Y_{n})$ of the total space $Y_{n}$ is given by the algorithm below.

For every $k\in\{1,\ldots, n\},$   we determine $\sigma_k \in \{ -1,0,1\}$ as follows. If $[\gamma_k]=0$  in $ \h_1 (\Sigma_{g}^b,\mathfrak{d} ; \R),$  then let $\sigma_k=0$. If $[\gamma_k] \not = 0$ then check if there exists a homology class $[x_k] \in \h_1 (\Sigma_{g}^b,\mathfrak{d}; \R)$ solving the following linear equation:
\begin{eqnarray}\label{eq:eq1} 
(\mathrm{Id}-t_{\gamma_k}^* \cdots t_{\gamma_1}^*)  ([x_k]) =[\gamma_k]
\end{eqnarray}
where $\mathrm{Id}$ is the identity element in the corresponding symplectic group.  If no such solution exists, let $\sigma_k=0.$ If there is a solution, let 
\begin{eqnarray}\label{eq:eq2}
\sigma_k= \mathrm{sign}(1+Q([\gamma_k], [x_k])).
\end{eqnarray}

Then the signature is given by
\[ 
\sigma(Y_{n})=-\displaystyle \sum_{k=1}^{n} \sigma_k - 
\# \{ k \in \{1,\ldots, n\}; [\gamma_k]=0  \; \text{ in } \h_1 (\Sigma_{g}^b,\mathfrak{d}; \R)\}.
\]
\end{theorem}

\begin{remark}\begin{enumerate}
\item During the course of our proof, we will show that any solution for the linear equation \eqref{eq:eq1} gives the same  $\sigma_k$. 
\item When $b=0$ the above algorithm is in fact a reformulation of Ozbagci's algorithm. Our $\sigma_k$ agrees with Ozbagci's local signature contributions for each vanishing cycle. A computer program written in SAGE implementing the above algorithm is available at our web sites. 
\item The above theorem can be modified to compute signatures of achiral Lefschetz fibrations. One just needs to change the sign of  $\sigma_k$'s whenever the vanishing cycle  $\gamma_k$ corresponds to an achiral singularity. 
\item We also  interpret $\sigma_k$ as the   Maslov ternary index of some naturally occurring Lagrangian subspaces which are graphs of some symplectic maps associated with the monodromy of the Lefschetz fibration, \cite{clm,bcrr}.   Equivalently, in the spirit of  \cite{bcrr},  we show that $\sigma_k$ corresponds to a special value of the Meyer's cocycle .
\item A systematic study of signatures of bordered Lefschetz fibrations is initiated by Miyamura, \cite{Miyamura}. When we extend our results to relative case we benefited a lot from his ideas. 
\end{enumerate}
\end{remark}

Next we give a formula that computes the signatures of finite cyclic branched covers of Lefschetz fibrations over disk where the branch locus is a regular fiber. 

 For any linear automorphism $\varphi$, the eigenspace corresponding to a real number $\lambda$ is denoted by $E_\lambda(\varphi)$, which could be the trivial vector space if $\lambda$ is not an eigenvalue. Consider a mapping class $\phi \in {\rm Mod}(\Sigma_{g}^b)$ of the surface. For any positive integer $m,$ we can define a symmetric bilinear pairing 
\begin{align}\label{eq:bra}
I_{\phi^{m+1}}: \frac{\h_1(\Sigma_g^b, \mathfrak{d};\R)}{E_{1}(\phi^{m+1}_{*})}\times \frac{\h_1(\Sigma_g^b, \mathfrak{d};\R)}{E_{1}(\phi^{m+1}_{*})}\rightarrow \R
\end{align}
\begin{align*}
I_{\phi^{m+1}}(z_1,z_2)&=Q(\phi^{m}_{*}(z_1),z_2)+Q(\phi^{m-1}_{*}(z_1),z_2)+\ldots+Q(\phi_{*}(z_1),z_2)\\
&-Q(z_1,\phi_{*}^m (z_2))-Q(z_1,\phi_{*}^{m-1} (z_2))-\ldots-Q(z_1,\phi_{*}(z_2))
\end{align*}
for $z_1,z_2 \in \frac{\h_1(\Sigma_g^b, \mathfrak{d};\R)}{E_{1}(\phi^{m+1}_{*})}. $
Denote the signature of $I_{\phi^{m+1}}$ by $\sigma_{\phi^{m+1}_{c}}.$
Let $f:Y\rightarrow \D^2$ be the Lefschetz fibration with regular fiber $\Sigma_g^b$ given by the monodromy factorization $\phi=t_l\dots t_1$. Let $\widetilde{f}:\widetilde{Y}\rightarrow \D^2$ be the  Lefschetz fibration with regular fiber $\Sigma_g^b$ given by the monodromy factorization $\phi^n=(t_l\dots t_1)^n$. Note that $\widetilde{Y}$ is the $n$-fold cyclic branched cover of $Y$ branched along a regular fiber.
\begin{theorem}\label{thm:bra}
We have
$$\sigma(\widetilde{Y})=n\sigma(Y)-\sum_{m=1}^{n-1} \sigma_{\phi^{m+1}_{c}}$$
\end{theorem}  
Notice that when $b\neq 0$, the  sum on right hand side gives signature invariant of the fibered link $\partial \Sigma_g^b$ due to Gordon, Litherland and Murasugi \cite{GLM}. One can compute this invariant  from the Seifert matrix using a formula similar to above \cite[pp. 383]{GLM}. It is interesting to observe that the same can be done using the homological monodromy.   

Finally we give an example of Lefschetz fibration with positive signature. Such examples can easily be constructed from Stein manifolds whose handlebody diagrams are on  Legendrian knots with large positive Thurston-Bennequin number and by taking the corresponding positive allowable Lefschetz fibration, but it is not clear whether this approach yields small fiber genus. The important property of our example is that we can achieve arbitrarily large positive signature with fiber genus one. Our result is optimal in the sense that Lefschetz fibrations with planar fibers are known to have non-positive signature \cite{Miyamura}. 

\begin{theorem}\label{thm:positive}
Given any $g \geq 1, b \geq 0, n > 0$, there exists a Lefschetz fibration over $\mathbb{D}^2$ with fiber genus $g$ and $b$ boundary components, having $3n$ singular fibers,  whose  total space has signature  $n$. 
\end{theorem}

The first example of a positive signature Lefschetz fibration over $\mathbb{D}^2$ with fiber genus one with one boundary component was found by Ozbagci which appeared in his unpublished notes. With our techniques, we are able to promote Ozbagci's example to arbitrary genus and boundary components. It remains an open question whether there exists a Lefschetz fibration over $\S^2$ with positive signature. Thanks to Endo's result, if such a Lefschetz fibration exists then it  cannot be hyperelliptic, hence its fiber genus must be at least three. Another open problem is  whether the number of vanishing cycles in our result is optimal. More precisely, can one have a Lefschetz fibration over $\mathbb{D}^2$ with signature $n>0$ having $k<3n$ singular fibers? An easy argument due to Ozbagci shows this is impossible for $n=1$, but the other cases are still open.  


\section{Preliminaries}\label{sec:sec2}

\subsection{Symplectic Vector Spaces}

Let $V$ be a real $2g$-dimensional vector space. A skew-symmetric, non-degenerate bilinear form $Q$ on $V$ is called 
a symplectic form, and the pair $(V, Q)$ is called a symplectic vector space. A basis $\{a_1,b_1, a_2, b_2 \ldots, a_g, b_g,\}$ of $V$ is called a symplectic basis if it satisfies
$Q(a_i,a_j)= Q(b_i,b_j)=Q(a_i,b_j)=\delta_{ij}$ for $1 \leq i,j \leq g.$

Consider $\R^{2g}$ with basis $\{ x_1,y_1,\ldots, x_g,y_g\}$.  Then with respect to the  standard symplectic form 
\[
Q_{\mathrm{std}}=\sum_{n=1}^{g} dx_i\wedge dy_i,
\]
$(\R^{2g},Q_{\mathrm{std}})$ is a symplectic vector space and $\{ x_1,y_1,\ldots, x_g,y_g\}$ is a symplectic basis. Conversely any symplectic vector space $(V, Q)$ has a symplectic basis giving an isomorphism  $(V, Q)\cong (\R^{2g},Q_{\mathrm{std}})$ for some $g$.
The linear symplectic group $\sp(2g,\R)$ is defined to be the group of linear automorphisms of $\R^{2g}$ preserving the symplectic form $Q_{\mathrm{std}}$. In terms of matrices 
\[
\sp(2g,\R)=\{ A \in GL(2g,\R)\,:\, A^TJA=J\}
\]
where $J$ is the $2g \times 2g$ matrix: 
 
\[J = 
\begin{bmatrix}
0 & 1 & 0 & 0 \quad \ldots \quad 0 & 0 \\
-1 & 0 & 0 & 0 \quad \ldots \quad 0 & 0 \\
0 & 0 & 0 & 1 \quad \ldots \quad 0 & 0 \\
0 & 0 &-1 & 0 \quad \ldots \quad 0 & 0 \\
\vdots  \\
0 & 0 & 0 & 0 \quad \ldots \quad 0 & 1 \\
0 & 0 &0 & 0 \quad \ldots  -1 & 0 
\end{bmatrix}.
\]
The same definition can be made for any symplectic vector space $(V,Q)$: The set of all automorphisms $\varphi:V\to V$ satisfiying $Q(\varphi(x), \varphi(y))=Q(x, y)$ is called the symplectic group of $(V,Q)$ and is denoted by $\sp(V,Q)$. The symplectic basis theorem gives rise to an isomorphism  $ \sp(V,Q) \cong\sp(2g,\R)$ where $g=\mathrm{dim}(V)/2$.


 A subspace $L$ of a symplectic vector space $(V,Q)$ is called Lagrangian if $Q(x,y)=0$ for every $x,y \in L$ and $\mathrm{dim}(L)=\mathrm{dim}(V)/2$. Symplectic automorphisms naturally give rise to Lagrangian subspaces. 
\begin{definition}
For a given symplectic automorphism $\varphi \in \sp(V,Q)$ of the symplectic vector space $(V,Q)$, the graph and respectively conjugate graph of $\varphi$ are subspaces  defined by
\begin{align*}
\mathrm{graph}(\varphi) &=\{(x,\varphi(x)) \; | \; x\in V \},\\
\widetilde{\mathrm{graph}}(\varphi)&=\{(\varphi (x),x) \; | \; x\in V \}.
\end{align*}
\end{definition}

Both $\mathrm{graph} (\varphi)$  and $\widetilde{\mathrm{graph}}(\varphi)$ are  Lagragian subspaces of $\R^{2g} \oplus \R^{2g}$ with the symplectic form $Q \oplus -Q$
\[
((x,y),(x', y') \mapsto Q( x,x') - Q( y,y').
\]

In our context, $V$ will mostly refer to the first homology group with real coefficients of a (possibly disconnected) surface, or the relative homology group of a surface with boundary.   The intersection form on the homology with $\Z$ coeffients induces a symplectic form on $V$. Note that reversing the orientation of the surface changes the sign of the symplectic form.


\subsection{Maslov Triple (Ternary) Index}

\begin{definition} \label{defn23}\cite{Lion,clm}
For a given three Lagrangians $A,B,C$ in a symplectic vector space $(V,{Q})$, the Maslov ternary index $\tau_V(A,B,C)$ is characterized by the following properties.

\begin{enumerate}

	\item\label{Maslov:i} \textit{Skew Symmetry}: 
	For a permutation $p$ of the three letters,
	\[
	\tau_V (p(A), p(B), p(C))=\mathrm{Sign}(p) \cdot \tau_V(A,B,C),
	\]
	\noindent where $\mathrm{Sign}(p)$ is the sign of the permutation $p$.
	\item \textit{Symplectic Additivity}: For three Lagrangians $A,B,C \subset V,$ and three Lagrangians $A^{'},B^{'},C^{'}\subset W,$
	\[
	\tau_{V\oplus W}(A \oplus A^{'},B \oplus B^{'},C \oplus C^{'})=\tau_V(A,B,C)+ \tau_W(A^{'},B^{'},C^{'}).
	\]
	
	\item \textit{Symplectic Invariance}:\label{Maslov:iii}
	For a symplectic automorphism $\varphi \in \sp(V,Q),$
	\[
	\tau_V(A,B,C)=\tau_V(\varphi(A),\varphi(B),\varphi(C)).
	\]
	
	\item \label{Maslov:iv}\textit{Normalization}: 
	For the Lagrangians $\R, \R(1+i),\R(i)$ in $\R^2=\C$ with the standard skew symplectic form $Q_{\mathrm{std}}$ in $\C$,
	\[
	\tau_{\C}(\R,\R(1+i),\R(i))=-1.
	\]

\end{enumerate}
\end{definition}
By \cite[Theorem 8.1]{clm}, there exists a unique system of functions $\tau_V(A,B,C)$ which satisfies the above Properties \eqref{Maslov:i} through \eqref{Maslov:iv}, so the  definition makes sense. One can think of the Maslov ternary index as a kind of ``cross ratio'' of triplets of Lagrangians in a symplectic vector space.   If $(A,B)$ and $(A',B')$ are two pairs of transverse Lagrangians in a symplectic vector space one can find a symplectic automorphism  $\varphi $ such that $\varphi(A)=A^{'}$ and $\varphi(B)=B^{'}$. However, the symplectic group does not act transitively on triples of Lagrangians. In fact The configuration of three transverse Lagrangians is completely determined by their index.

It was also shown in \cite{clm} that the Maslov ternary index is the same as Wall's signature defect which we review now. Let  
\begin{align}\label{eqn:W}W=W_V(A,B,C)= \displaystyle \frac{B\cap(C+A)}{(B\cap C)+ (B\cap A)}.
	\end{align}
Then a bilinear map $\Psi^{'}:B\cap(C+A) \times B\cap(C+A) \rightarrow \R$ is defined by $\Psi^{'}(b,b^{'})=Q(b,c^{'})$ where $a^{'}+b^{'}+c^{'}=0$ for some $a^{'}\in A, c^{'}\in C.$ Moreover any other choice of $a^{'}$ and $c^{'}$ yields the same value, so $\Psi^{'}$ is a well-defined bilinear map on $B\cap(C+A).$ It is easy to see that $\Psi^{'}(b,b^{'})=0$ if $b$ or $b^{'}$ is in $B\cap C+B\cap A.$ Hence $\Psi^{'}$ gives rise to a symmetric nonsingular bilinear map $\Psi$ on $W.$  Then the signature of $\Psi$ equals the Maslov ternary index $\tau_V(A,B,C)$.  Note that the $\tau_V(A,B,C)=0$ whenever $\mathrm{dim}\,W=0$.  This simple observation  will be very useful throughout the paper.


\subsection{Lefschetz fibrations}
Let $Y$ be a compact, oriented smooth $4$-manifold. A Lefschetz fibration on $Y$ is a smooth surjective map $f:Y\rightarrow B$ such that:
\begin{enumerate}
	\item $\{ b_1,b_2, \ldots, b_n \}$ are the critical values of $f$ inside $B$ with $p_i \in f^{-1}(b_i)$ a unique critical point of $f$  for each $i,$ and
	\item about each $b_i$ and $p_i,$ there are local complex coordinate charts agreeing with the orientations of $Y$ and $B$ such that locally $f$ can be expressed as $f(z_1,z_2)=z_1^2+z_2^2.$
\end{enumerate}
A bordered Lefschetz fibration is a Lefschetz fibration over $\D^2$ where the fibers have non-empty boundary.  The boundary of a  bordered Lefschetz fibration defines an open book structure.

\begin{definition} 
An open book structure (or decomposition) of a $3$-manifold $X$ is a surjective map  $\pi: X \rightarrow \D^2$ such that $\pi^{-1}(\text{int}\D)$ is a disjoint union of solid tori and 
\begin{enumerate}
	\item $\pi$ is a fibration over $\S^1=\partial \D^2.$
	\item On each of the solid tori, $\pi$ is the projection map $\S^1 \times \D^2 \rightarrow\D^2.$ The centers of these solid tori are called binding.
\end{enumerate}
	So in the complement of the binding $\pi$ is a fibration over $\S^1,$ with fibers surfaces with boundary. Alternatively, we can define an open book structure on $X$ just in terms of its pages and the monodromy $(\Sigma_g^b, \phi)$ of the fibration $\pi: X\rightarrow \S^1.$ Here $\Sigma_g^b$ is an oriented surface with non-empty boundary and $\phi:\Sigma_g^b\rightarrow \Sigma_g^b$ is a diffeomorphism which is identity on $\partial \Sigma_g^b$. Then $X$ is the union of the mapping torus of $\phi$ and solid tori glued along their boundaries in the obvious way. The monodromy of the open books arising from the boundary of bordered Lefschetz fibrations  is a composition of right handed Dehn twists (coming from the vanishing cycles).

\end{definition} 

\begin{definition} Given any bordered Lefschetz  fibration $f:Y\to\mathbb{D}^2$ with fiber $\Sigma_g^b$ with $b\geq 1$, we  define its closure,  which is a new Lefschetz fibration $\overline{f}:\overline{Y}\to\mathbb{D}^2$  with closed fibers of genus $g+b-1$, as follows. We first fiberwise attach $2$-dimensional $1$-handles along different components $\partial \Sigma_g^b$ reducing the number of boundary components to one, trivially extending the monodromy over the one-handles.   Then cap off the remaining boundary component by a disk. 
\end{definition}
  
 Notice that  when $b\geq 1$ the isomorphism $\h_1(\Sigma_g^b,\mathfrak{d};\mathbb{R})\cong \h_1(\Sigma_{g+b-1};\R)$ and the natural homomorphism $\mathrm{Mod}(\Sigma_g^b)\to \mathrm{Mod}(\Sigma _{g+b-1})$ identifies  the symplectic representations of the monodromies of the original Lefschetz fibration and its closure. More precisely the basis of  $\h_1(\Sigma_g^b,\mathfrak{d};\mathbb{R})$ given on the left side of Figure \ref{fig:basis}  naturally corresponds to the basis  of $ \h_1(\Sigma_{g+b-1};\R)$ and the homological actions of any element in $\mathrm{Mod}(\Sigma_g^b)$ to these basis are the same.

\begin{figure}[h]
	\includegraphics[width=.40\textwidth]{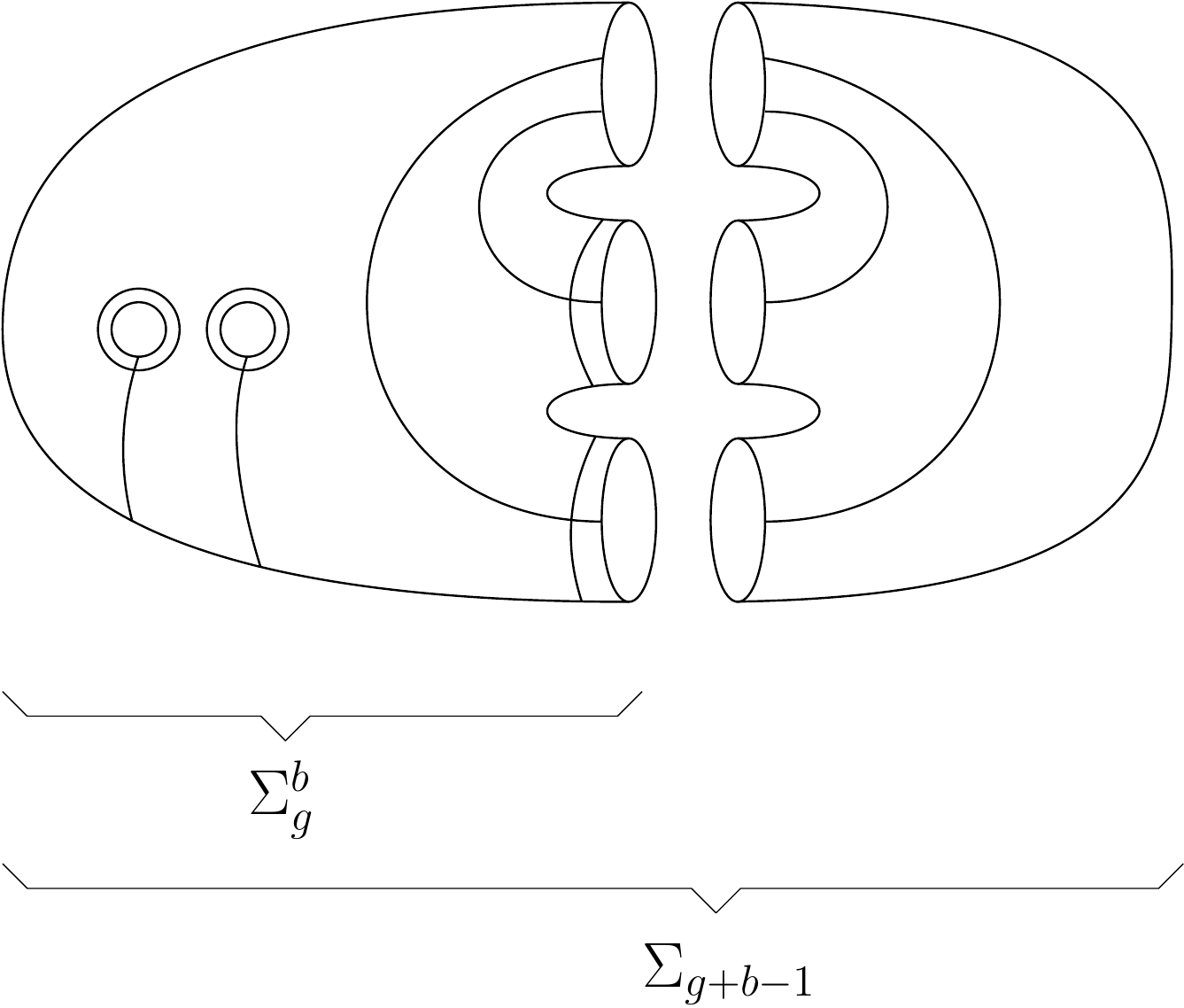}
	\caption{Closing up  relative homological bases of bordered surfaces}
	\label{fig:basis}
\end{figure}


\subsection{Signature and Wall's formula} 

Let $Y_1$ and $Y_2$ be two compact oriented $4$-manifolds. If a closed manifold $Y$ is obtained by gluing along the whole boundaries of $Y_1$ and $Y_2$ via an orientation reversing diffeomorphism, then the signature of $Y$ is equal to sum of the signatures of $Y_1$ and $Y_2.$ This is Novikov additivity \cite{as}.  On the other hand it is also possible to glue manifolds $Y_1$ and $Y_2$ along  common submanifold $X_0$ in their boundary where $X_0$ may itself have boundary. In this case, the resulting manifold $Y$ is a $4$-manifold with boundary and its signature is not the sum of the signatures of $Y_1$ and $Y_2$, the defect in this argument is a Maslov ternary index\cite{wall}.

\subsection{Wall's Non-additivity Formula}\label{subsec:Wallnonadd} 

Let $Y, Y^{-}, Y^{+}$ be $4$-manifolds, $X^{0},X^{-}, X^{+}$ be $3$-manifolds and $Z$ be a $2$-manifold such that $Y=Y^{-} \cup Y^{+},$ and $Y^{-} \cap Y^{+}=X^{0}.$ Moreover
\begin{align*}
\partial Y^{-}&=X^{-}\cup X^{0}, \\
\partial Y^{+}&=X^{0}\cup X^{+},\\
\partial X^{-}&=\partial X^{0}=\partial X^{+}=Z.
\end{align*}

Suppose $Y$ oriented, inducing orientations on $Y^{-}$ and $Y^{+}.$  The rest is oriented as follows:
\begin{align*}
\partial_{*} [Y^{-}]&=[X^{0}]-[X^{-}],\\ 
\partial_{*} [Y^{+}]&=[X^{+}]-[X^{0}],\\ 
\partial_{*} [X^{-}]&=\partial_{*}[X^{0}]=\partial_{*}[X^{+}]=[Z].
\end{align*}

\begin{figure}[h]
	\includegraphics[width=0.35\textwidth]{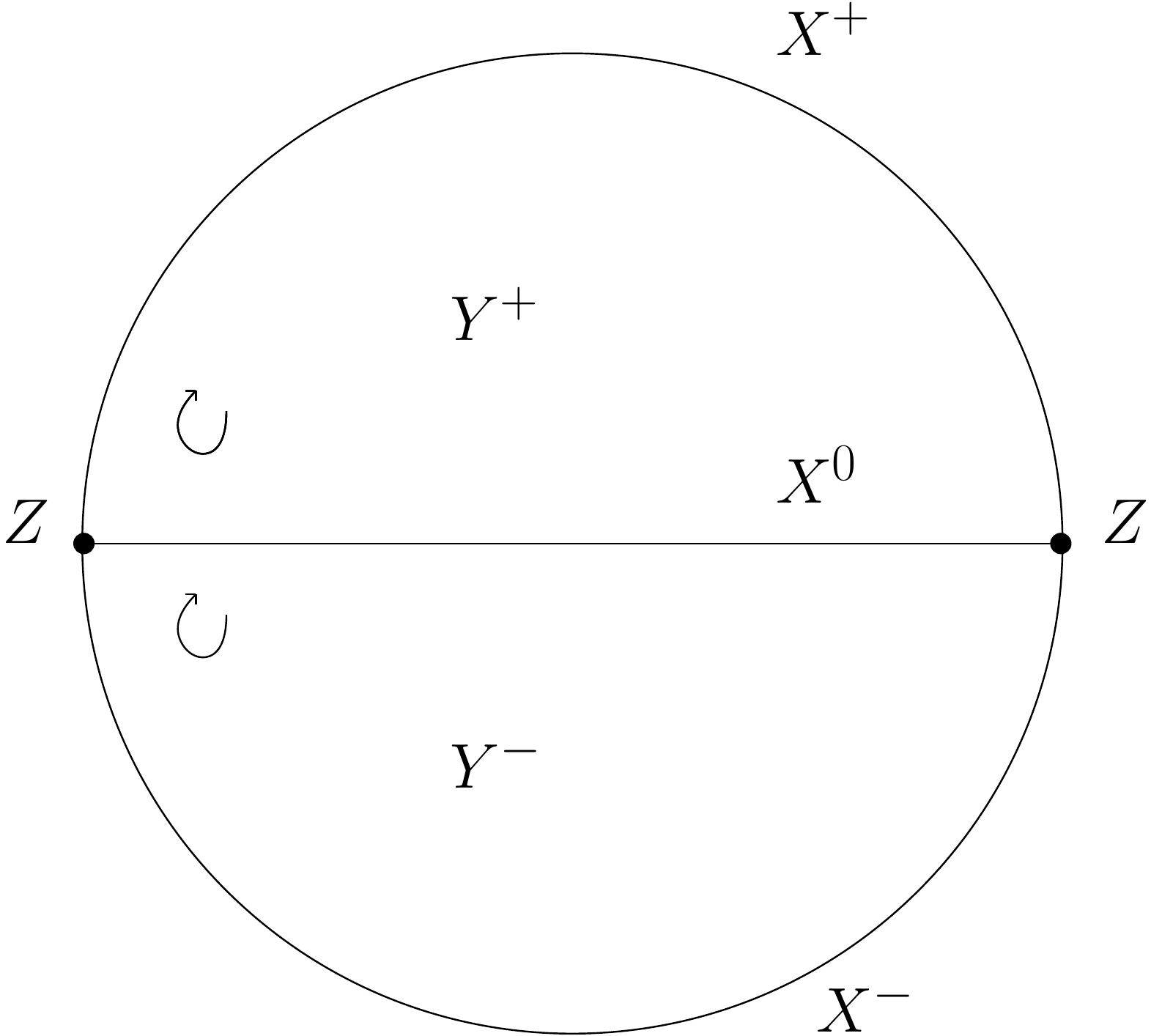}
	\caption{ }
	\label{fig:den}
\end{figure}

Let $V=\h_1(Z;\R)$.  Define the following subspaces of $V$
\begin{align}
\label{eqn:Wallset1}	A&=\text {ker }(i_{*}: \h_1(Z;\R) \rightarrow \h_1(X^{-};\R), \;  \\
\label{eqn:Wallset2}	B&=\text {ker }(i_{*}:\h_1(Z;\R) \rightarrow \h_1(X^{0};\R), \;  \\
\label{eqn:Wallset3}	C&=\text {ker }(i_{*}:\h_1(Z;\R) \rightarrow \h_1(X^{+};\R),
\end{align}
Then $\mathrm{dim} \, A= \mathrm{dim} \, B =\mathrm{dim} \, C= \frac{1}{2}\mathrm{dim} \, V.$ Denote the  symplectic intersection form of $V$ by $Q$ and $Q(A\times A)=Q(B\times B)=Q(C\times C)=0$. The subspaces $A,B$ and $C$ are Lagrangian subspaces  for $Q$. Then we have

\begin{theorem}\cite{wall}\label{theorem:2}
$\sigma(Y)=\sigma(Y^{-})+\sigma(Y^{+})-\tau_V(A,B,C).$
\end{theorem}

\section{Signature of Closure}

The following statement allows us to reduce the proofs of our result to their respective special cases where the fibers are closed.
\begin{theorem} \label{thm:clos}
Signature of the total space of a bordered Lefschetz fibration is equal to that of its closure. 
\end{theorem}
\begin{proof}
We will use standard topological arguments about Lefschetz fibrations. For details, the reader can consult \cite{MR1891172} where our method already appeared in.   Attaching  $2$-dimensional $1$-handles to fibers corresponds to attaching $4$-dimensional $1$-handles to the total space, so does not change the signature. Hence it suffices to prove the statement in the case where fibers have connected boundary.

Suppose $f^{+}:Y^{+}\rightarrow \D^2$ is a bordered Lefschetz fibration with regular fiber $\Sigma_{g}^1$ which is a compact genus $g$ surface with one boundary component. Our aim is to compute $\sigma(Y^+).$ We study the construction of the closure $f$ of $f^+$  carefully. 
The restriction of $f^+$ to $\partial Y^{+}$ is an open book decomposition with a connected binding $K$ which is just a copy of  $\partial \Sigma_{g}^1$ in $\partial Y^{+}.$ The complement of a neighborhood of $\mathcal{N}(K)$ is a surface bundle over $\S^1$ with fibers $\Sigma_{g}^1$.
We can cap off the fibers of $f^+$ by attaching a $2$-handle $Y^{-}=\D^2\times\D^2$ attached along $\partial \D^2 \times \D^2$ to $\mathcal{N}(K)\subseteq \partial Y^{+}$ in such a way that the circles $\partial \D^2\times\{p\}$ are identified with the boundaries of the fibers $\Sigma_{g}^1.$ In other words the $2$-handle is attached along the binding with the page framing. The result is a Lefschetz fibration $f:Y\rightarrow \D^2$ where the fibers $\Sigma_g$ are now closed surfaces of genus $g.$ We write $Y=Y^+ \cup Y^-.$ Let 
\begin{align*}
X^{0}&=\mathcal{N}(K)=\partial \D^2 \times \D^2 \\
X^{+}&=(\partial Y^+) \setminus \mathcal{N}(K) \\
X^{-}&= (\partial Y^-) \setminus (\partial \D^2\times \D^2) \\
Z&=\partial \mathcal{N}(K) = \partial \D^2\times \partial \D^2
\end{align*}
Note that $Z$ is a torus. We refer to homology generators $\mu=(\{p\} \times \partial \D^2)$ and $\ell=(\partial \D^2 \times \{p\})$ in $V=\h_1(Z;\R)$ the meridian and the longitude respectively of the torus $Z.$
Apply Wall's formula to the above setting, equations \eqref{eqn:Wallset1}, \eqref{eqn:Wallset2}, and \eqref{eqn:Wallset3} give $A=\langle \ell \rangle$, $  
B=\langle \mu \rangle$,   $C=\langle \ell \rangle$. Then $B\cap(C+C)=\{0\}$ implying $\text{dim} \, W=0.$ Hence the Maslov  index $\tau_V(A,B,C)$ vanishes. Since $\sigma(Y^{-})=0,$ Theorem \ref{theorem:2} gives $\sigma(Y)=\sigma(Y^{+}).$
\end{proof}


\section{Wall's Formula and Partial Fiber Sum Decompositions}
 Suppose we are given a (bordered) Lefschetz fibration $f:Y\rightarrow \D^2$  with  regular fiber $\Sigma_{g}^b$ and monodromy $\phi$. If necessary by postcomposing $f$ with a diffeomorphism of a disk,
we can always identify the base disk $\D^2$ with the unit disk  $\{(x,y)\,:\, x^2+y^2\leq 1\} $ in $\R^2$. Without loss of generality, we assume there exists no singular value of $f$ on the line segment $ [-1,1]\times \{0\}.$ 
Define
\begin{eqnarray*}
D^{+}&=&\{(x,y)\in \D^2\,:\,y\geq 0 \} \\
D^{-}&=&\{(x,y)\in \D^2\,:\,y\leq 0 \} 
\end{eqnarray*}
so that $\D^2=D^{+}\cup D^{-}$ and $Y=Y^{+}\cup Y^{-}$ where  
$f^{\pm}=f|_{Y^{\pm}}$ and $Y^{\pm}={f^{\pm}}^{-1}(D^{\pm}).$ The pair $(f^{+},f^{-})$ is called a partial fiber sum decomposition of the Lefschetz fibration $f$. We denote by $\phi^+$ and $\phi^-$ the monodromies of $f^+$ and $f^-$ respectively. Let 
 \begin{align}
	V&=\h_1(\Sigma_g^b,\mathfrak{d};\R) \oplus \h_1( \Sigma_g^b,\mathfrak{d};\R), \label{eqn:partfibsum1}
\end{align}
\noindent equipped with the intersection form $Q\oplus-Q$  where $Q$ is the intersection form on $\h_1(\Sigma_g^b,\mathfrak{d};\R)$.  We define three Lagrangian subspaces of $V$ by
\begin{align}
	A&=\mathrm{graph}({\phi_{*}^{-}})\label{eqn:partfibsum2},\\
	B&=\mathrm{graph} (\text{id})\label{eqn:partfibsum3}, \\
	C &= \widetilde{\mathrm{graph}}({\phi_{*}^{+}}). \label{eqn:partfibsum4}
\end{align}

\begin{theorem}\label{thm:partfibsum} For the partial fiber sum decomposition $f=(f^+,f^-)$,
	\begin{equation}\label{eqn:partfibsum}
	\sigma(Y)=\sigma(Y^+)  + \sigma(Y^-)- \tau_V(\mathrm{graph}({\phi_{*}^{-}}), \mathrm{graph} (\text{id}), \widetilde{\mathrm{graph}}({\phi_{*}^{+}})).
	\end{equation}
	
	\end{theorem}
\begin{proof}

We will use Wall's non-addivity formula discussed in Section \ref{subsec:Wallnonadd} to understand how signature behaves under partial fiber sum decompositions.  First  assume that fibers are closed, so $b=0$.   Define 
\begin{align*}
X^{0}&=f^{-1}([-1,1]\times \{0\}), \\
X^{-}&=f^{-1}(D^-\cap \partial \D^2), \\
X^{+}&=f^{-1}(D^+\cap \partial \D^2), \\
Z&= \partial X^{-}= \partial X^{0}= \partial X^{+}=f^{-1}(\{-1,1\}\times\{0\})=\Sigma_g\coprod -\Sigma_g.
\end{align*}
 When we choose an orientation on $Y$, it induces orientations on $Y^{+}$ and $Y^{-}$. Then orient $X^{0},X^{+}$ and $X^{-}$  as in Wall's formula.  According to the above setting $V$, $A$, $B$, and  $C$ are given as in \eqref{eqn:partfibsum1},\eqref{eqn:partfibsum2}, \eqref{eqn:partfibsum3}, and \eqref{eqn:partfibsum4} respectively. The result follows from Theorem \ref{theorem:2}.

To prove the theorem when the fibers $\Sigma_g^b$ are bordered, $b\geq 1$, we first take the closure of our Lefschetz fibration and apply partial fiber sum decomposition the same way we do the bordered Lefschetz fibration. Noticing that  $(\overline{f})^{\pm}=\overline{f^{\pm}}$, the result follows from Theorem \ref{thm:clos} and the case $b=0$.
\end{proof}

\begin{remark}Even though we do not need this for the rest of the paper, we would like to point out the relationship between our  signature defect corresponds to Meyer's 2-cocycle $\mathrm{Mey}:Sp(2g,\mathbb{Z})\times Sp(2g,\mathbb{Z})\to \mathbb{Z}$. It is known that 
$$
\mathrm{Mey}(\phi^-_*,\phi^+_*)=-\tau_V(\mathrm{graph}({\phi_{*}^{-}}), \mathrm{graph} (\text{id}), \widetilde{\mathrm{graph}}({\phi_{*}^{+}})).
$$ 
\end{remark}
\section{Proof of the Main Theorem}

In this section we prove our main theorem assuming that fibers are closed. This is sufficient by Theorem \ref{thm:clos}. The proof follows by induction on number of vanishing cycles. 

\subsection{Base step} 
This is the case when a Lefschetz fibration has only one singular fiber. It is well known that the signature of the total space of this Lefschetz fibration is $0$ if its unique vanishing cycle is non-separating and is -$1$ if the vanishing cycle is separating. See for example Ozbagci in \cite{ozbagci}. For convenience of the reader, we include the proof here:
When a Lefschetz fibration over disk has no singular fiber, then its total space is $\Sigma_g\times \D^2$ which has signature $0,$ because the second homology is generated by a regular fiber which has self-intersection $0.$ To get a singular fiber we attach a -$1$-framed $2$-handle to a curve $\gamma$ on the fiber. If $\gamma$ is non-separating, then the homotopy type of the total space is $\Sigma_{g-1} \bigvee \S^1$ and its second homology is again generated by homology class of a regular fiber which has self intersection $0,$ hence the total space has signature $0.$ When $\gamma$ is separating, the homotopy type of the total space is $\Sigma_k \bigvee \Sigma_{g-k}$ for some $1\leq k\leq g-1,$ so the second homology has rank $2.$ In addition to the homology class of a regular fiber, a second generator comes from a subsurface bounded by $\gamma$ in a regular fiber and the core of the $2$-handle. Since the $2$-handle has framing -$1,$ the resulting class has self intersection -$1.$ The intersection form in this basis is 
$\begin{bmatrix}
	0 & 0 \\
	0 & -1
\end{bmatrix},$
so signature is -$1.$

\hspace{1cm}

\subsection{Inductive step} In this step we use a partial fiber sum decomposition to reduce the number of Dehn twists appearing in the monodromy factorization of Lefschetz fibrations. For every $k \in \{1, \dots, n\}$ we are given a Lefschetz fibration $f_{k}:Y_{k}\rightarrow \D^2$  with regular fiber $\Sigma_{g}$ and monodromy $t_{\gamma_k} \cdots t_{\gamma_1}.$  In what follows we will consider a special kind of decomposition so that $D^{-}$ contains only one Lefschetz critical value corresponding to $\gamma_k$ and $D^{+}$ contains all the others, i.e., the $\gamma_i$'s where $1 \leq i \leq k-1.$ Hence $Y_{k}^{+}$ is diffeomorphic to $Y_{k-1}$, for all $k=2,\ldots, n.$

 Suppose $\phi_{k}^{\pm}:\Sigma_g \rightarrow \Sigma_g$ is the monodromy of $f_{k}^{\pm}.$ From our choices $\phi_{k}^{-}=t_{\gamma_k}$ and $\phi_{k}^{+}=t_{\gamma_{k-1}} \cdots t_{\gamma_1}$. Let $V=H_1(\Sigma_g;\R)\oplus  H_1(\Sigma_g;\R)$ equipped with the symplectic form $Q\oplus -Q$ where $Q$ is intersection form on $\Sigma_g$. Let $A_k$, $B_k$ and $C_k$ denote the subspaces $\mathrm{graph}({\phi_{k\,*}^{-}})$, $\mathrm{graph} (\mathrm{id})$, and $\widetilde{\mathrm{graph}}({\phi_{k\,*}^{+}})$ respectively.  We will show that $\sigma_{k}$ given in the statement of the theorem is equal to  $\tau_V(A_k, B_k, C_k)$, for all $k=2,\ldots, n$. This finishes the proof because the  signature formula for partial fiber sums implies
\begin{align*}
\sigma(Y_n)=\sigma(Y_{n-1})+\sigma(Y_n^{-})-\sigma_n,
\end{align*}
by base step 
\[
\sigma(Y_{n}^{-})=
\begin{cases} 
      0 & \text{if} \; \gamma_n \; \text{is non-separating} \\
      -1 & \text{if} \; \gamma_n \; \text{is separating},
   \end{cases}
\]
and by inductive step assumption we have
\[
\sigma(Y_{n-1})={-\sum_{j=1}^{n-1} \sigma_{j} } -\# \{ j \in \{1,\ldots, n-1\}; [\gamma_j]=0  \; \text{ in } \h_1(\Sigma_{g};\R)\}.
\]

\subsection{Identifying the local signatures} It remains to prove $\sigma_k=\tau_V(A_k,B_k,C_k)$. 
\begin{lemma}
Assume $\gamma_k$ is a nonseparating curve then $\mathrm{dim}(E_{1}(\phi^{-}_{k\,*}))=2g-1.$
\end{lemma}

\begin{proof}
The surface $\Sigma_g \setminus \gamma_k$ is of genus $g-1$ and has two boundary components. Let $\alpha_1,\ldots,\alpha_{2g-2}$ be a homology basis for $\h_1(\Sigma_g \setminus \gamma_k;\R),$ then they are linearly independent in $\h_1(\Sigma_g \setminus \gamma_k;\R)$ as well. Clearly $\phi^{-}_{k\,*}(\alpha_j)=\alpha_j$ for $\forall j=1,\ldots, 2g-2.$ Also $\phi^{-}_{k\,*}(\gamma_k)=\gamma_k.$ So $\gamma_k,\alpha_1,\ldots,\alpha_{2g-2}$ is a basis for $E_{1}(\phi^{-}_{k\,*}).$

\end{proof}

Let $W_k=W_V(A_k,B_k,C_k)$ as Equation \eqref{eqn:W}, and let $\widetilde{\gamma_{k}}$ be a dual of $\gamma_k.$ i.e., $Q(\gamma_k, \widetilde{\gamma_{k}})=1.$ Then
\begin{align}\label{eq:eq6}
\phi_{k\,*}^{-}(\widetilde{\gamma_{k}})=\gamma_k+\widetilde{\gamma_{k}} 
\end{align}
which implies $\mathrm{dim}\left(\frac{\h_1(\Sigma_g;\R)}{E_{1}(\phi^{-}_{k\,*})}\right)=1$ and the coset containing $\widetilde{\gamma_{k}}$ is a basis for $\frac{\h_1(\Sigma_g;\R)}{E_{1}(\phi^{-}_{k\,*})}.$
The map $x\rightarrow (x,x)$ gives rise to an isomorphism between $\frac{\h_1(\Sigma_g;\R)}{E_{1}(\phi^{-}_{k\,*})}$ and $\frac{B_k}{B_k\cap A_k}.$ Since $W_k$ is isomorphic to a subspace of a quotient space of $\frac{B_k}{B_k\cap A_k},$ the space $W_k$ is at most $1$-dimensional and is generated by the coset containing $(\widetilde{\gamma_{k}},\widetilde{\gamma_{k}}).$

Therefore it sufficies to figure out $\Psi_{W_k}((\widetilde{\gamma_{k}},\widetilde{\gamma_{k}}), (\widetilde{\gamma_{k}},\widetilde{\gamma_{k}})).$
From the description of  $\Psi_{W_k}$ in Section \ref{sec:sec2}, first look for $z_k$ and $x_k\in \h_1(\Sigma_g;\R)$ solving   
\begin{equation}\label{eq:int}
(z_k,\phi^{-}_{k\,*}z_k)+(\widetilde{\gamma_{k}}, \widetilde{\gamma_{k}})+(\phi^{+}_{k\,*}x_k,x_k)=0
\end{equation}
where $(z_k,\phi^{-}_{k\,*}z_k)\in A_k,$ $(\widetilde{\gamma_{k}}, \widetilde{\gamma_{k}}) \in B_k$ and $(\phi^{+}_{k\,*}x_k,x_k)\in C_k.$ Then
\begin{align}
z_k+\widetilde{\gamma_{k}}+\phi^{+}_{k\,*}x_k&=0 \label{eq:eq7},  \\
\phi^{-}_{k\,*}z_k+\widetilde{\gamma_{k}}+x_k&=0 \label{eq:eq8}.
\end{align}

In the equation (\ref{eq:eq8}) if we substitute $z_k$ from the equation (\ref{eq:eq7}), we obtain
\begin{align}\label{eq:eq9}
(\mathrm{Id}-\phi^{-}_{k\,*}\phi^{+}_{k\,*})\cdot x_k=\gamma_{k}.
\end{align}
If no such $x_k$ exists, then this means $(\widetilde{\gamma_{k}}, \widetilde{\gamma_{k}})$ is not in $B_k\cap(C_k+A_k).$ So dim $W_k=0$ implying $\tau_V(A_k,B_k,C_k)=0.$ Conversely if $x_k$ is any solution to the above equation then we let $z_k=-\widetilde{\gamma_{k}}-\phi^{+}_{k\,*}x_k$ from the equation (\ref{eq:eq7}), and see that equation \eqref{eq:int} has a solution.  Then by \eqref{eq:int},
\begin{align*}
 \Psi_{W_k}((\widetilde{\gamma_{k}}, \widetilde{\gamma_{k}}),(\widetilde{\gamma_{k}}, \widetilde{\gamma_{k}})) &= Q(\widetilde{\gamma_{k}},\phi^{+}_{k\,*}x_k)-Q(\widetilde{\gamma_{k}},x_k).\\
&= Q(\phi^{-}_{k\,*}\widetilde{\gamma_{k}},\phi^{-}_{k\,*}\phi^{+}_{k\,*}x_k) - Q(\widetilde{\gamma_{k}},x_k).\\
&= Q(\gamma_{k}+\widetilde{\gamma_{k}},x_k-\gamma_{k}) - Q(\widetilde{\gamma_{k}},x_k)\\
&= Q(\gamma_{k},x_k)+Q(\widetilde{\gamma_{k}},x_k)-Q(\widetilde{\gamma_{k}},\gamma_{k})-Q(\widetilde{\gamma_{k}},x_k)\\
&=Q(\gamma_{k},x_k)+1
\end{align*}
 which shows that $\tau_V(A_k, B_k, C_k)=\sigma_k$ as given in Equation (\ref{eq:eq2}). Here we used the fact that  $\phi^{-}_{k\,*}\in \sp(2g;\R)$ (so it preserves the intersection form $Q$), as well as  Equations \eqref{eq:eq6} and \eqref{eq:eq9}.


\section{Some Computational Shortcuts}
In this section, we will prove some lemmas which will be useful in our computations. Throughout we assume the fibers of our Lefschetz fibrations are closed. For the following lemma, we continue using the notation in the previous section.  

\begin{lemma}\label{lemma:l4}
Suppose $\gamma_{k}$ is non-separating. If $\phi^{+}_{k\,*}\widetilde{\gamma_{k}}=\widetilde{\gamma_{k}}$  then $\sigma_k=0.$ i.e., if the monodromy of a Lefschetz fibration preserves the dual curve of $\gamma_k,$ adding a vanishing cycle on $\gamma_k$ does not change the signature.
\end{lemma}

\begin{proof}
We know $B_k\cap A_k$ is isomorphic to $E_{1}(\phi^{-}_{k\,*})$ which has codimesion one, and the quotient space $\frac{\h_1(\Sigma_g;\R)}{E_{1}(\phi^{-}_{k\,*})}$ is generated by the coset containing $\widetilde{\gamma_{k}}.$ We also know that $B_k\cap C_k$ is isomorphic to $E_{1}(\phi^{+}_{k\,*}).$ The latter contains the homology class of $\widetilde{\gamma_{k}}$ by our assumption. Hence 
\begin{align}
E_{1}(\phi^{+}_{k\,*})+E_{1}(\phi^{-}_{k\,*})=\h_1(\Sigma_g;\R)
\end{align}
and $W_k$ is isomorphic to a subspace of the $0$-dimensional space $\frac{\h_1(\Sigma_g;\R)}{E_{1}(\phi^{+}_{k\,*})+E_{1}(\phi^{-}_{k\,*})}$.
Therefore $W_k$ is itself $0$-dimensional implying $\sigma_k=0.$

\end{proof}

The next lemma considers an arbitrary partial fiber sum decomposition of a Lefschetz fibration whose monodromy is homologically trivial.

\begin{lemma}
Let $f:Y \rightarrow \D^2$ be a Lefschetz fibration with monodromy acting trivially on homology of $\Sigma_g.$ Suppose we applied a partial fiber sum decomposition $f=(f_+,f_-)$ with the corresponding total space $Y^+$ and $Y^-,$ then
$\sigma(Y)=\sigma(Y^{-})+\sigma(Y^{+}).$ 
\end{lemma}
\begin{proof}
Let $\phi^{\pm}$ denote the monodromy of $f^{\pm}.$ Then $\phi^{-}_{*}\phi^{+}_{*}=\mathrm{Id}.$ By permuting the roles of $A,B$ and $C$, we will show that $W=W_V(B,A,C)$ is $0$-dimensional.  Suppose $(x,\phi^{-}_{*}(x))\in A$ and $(\phi^{+}_{*}(z), z)\in C$ then $\phi^{-}(\phi^{+}z)=z$ which implies $z\in E_1(\phi^{-}_{*}\phi^{+}_{*}).$ Then $A\cap C= \{(\phi^{+}_{*}z,z) : z\in E_1(\phi^{-}_{*}\phi^{+}_{*})\}.$
Since $\phi^{-}_{*}\phi^{+}_{*}=\mathrm{Id}$ and $A\cap C $ is $2g$ dimensional, from the definition $A\cap (B+C)$ is at most $2g$ dimension. Hence $W$ is $0$ dimensional, so the Maslov index $\tau_V(\mathrm{graph}({\phi_{*}^{-}}), \mathrm{graph} (\text{id}), \widetilde{\mathrm{graph}}({\phi_{*}^{+}}))$ vanishes.

\end{proof}

\begin{remark}
If the monodromy is identically trivial, the above lemma can also be proved using Novikov additivity.
\end{remark}


\section{Taking Exponents}
In this section we shall prove Theorem \ref{thm:bra}. Let $f:\widetilde{Y}\rightarrow \D^2$ be a Lefschetz fibration with regular fiber $\Sigma_g^b$ and  monodromy $\phi^{m+1}.$ Consider a partial fiber sum decomposition $(f^+,f^-)$ where $f^+$ and $f^-$ are Lefschetz fibrations with monodromy $\phi^-=\phi$ and $\phi^+=\phi^m.$ The manifolds $Y^{\pm}$ corresponds to $(f^{\pm})^{-1}(D^{\pm})$ as before.
\begin{theorem}\label{theorem:t5}
We have 
\[
\sigma(Y)=\sigma(Y^+)+\sigma(Y^-)-\sigma_{{\phi^{m+1}_{c}}}.
\]
\end{theorem}
\begin{proof}
We apply Theorem \ref{thm:partfibsum} to the aforementioned partial fiber sum decomposition. We have $V=H_1(\Sigma_g^b,\mathfrak{d};\R)\oplus  H_1(\Sigma_g^b,\mathfrak{d};\R)$ equipped with the symplectic form $Q\oplus -Q$ where $Q$ is intersection form on $\Sigma_g^b$, $A =\mathrm{graph}(\phi_*)$,  $B =\mathrm{graph}(\mathrm{Id})$,  $C =\widetilde{\mathrm{graph}}(\phi^m_*)$, and $W=W_V(A,B,C)$.

We must show that $\sigma_{\phi^{m+1}_{c}}=\tau_V(A,B,C).$ From the definitions, it is clear that 
\begin{align*}
B\cap A=\{(x,x):x \in E_1(\phi_{*})\}\subseteq B\cap C=\{(x,x):x \in E_1(\phi_{*}^m)\}.
\end{align*}
\begin{align*}
B\cap (C+A)&= \{(\phi_{*}(x)+z, \; \phi_{*}(x)+z) : \phi_{*}^m(z)-z=\phi_{*}(x)-x\}.
\end{align*}
Letting $y=(\mathrm{Id}+\phi_{*}+\ldots +\phi_{*}^{m-1})(z)-x$ in the above description, we see that 
\begin{eqnarray*}
\phi_{*}(y)&=& \phi_{*}(z)+\phi_{*}^2(z)+\ldots +\phi_{*}^m(z)-\phi_{*}(x)\\
&=&\phi_{*}(z)+\phi_{*}^2(z)+\ldots +\phi_{*}^{m-1}(z)+z-x \\
&=&y.
\end{eqnarray*}
Hence $y \in E_1(\phi_{*}),$ and we have
\begin{align*}
B\cap (C+A)&= \{(z+\phi_{*}(z)+\ldots+\phi_{*}^m(z)-y, z+\phi_{*}(z)+\ldots+\phi_{*}^m(z)-y) :\\
 &\hspace{5cm}y,z \in \h_1(\Sigma_g^b,\mathfrak{d};\R) \; \text{and} \; y \in E_1(\phi_{*})\}.
\end{align*}
Hence the map $z \rightarrow (z+\phi_{*}(z)+\ldots+\phi_{*}^m(z), \; z+\phi_{*}(z)+\ldots+\phi_{*}^m(z))$ gives rise to an isomorphism from $\frac{\h_1(\Sigma_g^b,\mathfrak{d};\R)}{E_1(\phi_{*}^{m+1})}$ to $W.$
It remains to identify the symmetric bilinear pairings $I_{\phi^{m+1}}$ and $\Psi$ under the above isomorphism.
We have
 \begin{align*}
&\Psi^{'}((z_1+\ldots +\phi_{*}^m z_1,\,z_1+\ldots+\phi_{*}^m z_1), (z_2+\ldots +\phi_{*}^m z_2,\,z_2+\ldots+\phi_{*}^m z_2))\\
&=-Q(z_1+\ldots+\phi_{*}^m z_1,\phi_{*}^m z_2)+Q(z_1+\ldots+\phi_{*}^m z_1, z_2)\\
 &=Q(z_1,z_2)+\ldots+Q(\phi_{*}^m(z_1),z_2)-Q(z_1, \phi_{*}^m(z_2))-\ldots-Q(\phi_{*}^m(z_1),\phi_{*}^m(z_2)) \\
&={I_{\phi^{m+1}}(z_1,z_2)}.
 \end{align*}

Here we use the fact that  $Q(z_1,z_2)=Q(\phi_{*}^m (z_1), \phi_{*}^m(z_2))$ as $\phi_{*}^{m}\in \sp(2g;\Z).$ Since the forms are identical, the result follows.
\end{proof}

For our computations we also need a matrix representation for the symmetric bilinear form $I_{\phi^{k+1}}$. Take a symplectic basis so that the intersection form on $H_1(\Sigma,\mathbb{R})$ is represented by the matrix $J$. We represent the linear map $\phi_*$ as a matrix using the same basis. Then by Equation \eqref{eq:bra}, 
\begin{align}\label{eq:eq10}
\sigma_{\phi^{k+1}_{c}}= \text{sign} \left( (\phi^T)^k \cdot J+ (\phi^T)^{k-1}\cdot J+ \cdots + \phi^T \cdot J - J\cdot \phi \cdots -J\cdot \phi^k \right). 
\end{align}

\begin{figure}[h]
	\includegraphics[width=0.40\textwidth]{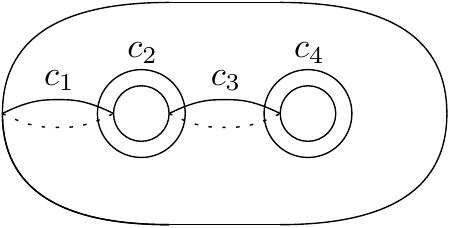}
	\caption{ The curves $c_1,c_2,c_3$ and $c_4$}
	\label{figure:3}
\end{figure}

\begin{example} Let $f:Y\rightarrow \D^2$ be a Lefschetz fibration with monodromy $(C_4C_3C_2C_1)^{10} =\mathrm{Id}$ where each $c_i \, (i=1,2,3,4)$ is a curve on genus-$2$ surface as in Figure \ref{figure:3}. It is known that the total space  of this Lefschetz fibration is homeomorphic (but not diffeomorphic) to $5\mathbb{CP}^2\sharp 29\overline{\mathbb{CP}^2}$ \cite{Matsumoto1,MR1624287}, so it has signature $-24$. We will verify this using our techniques. Let $C_i$ denote positive Dehn twist along each $c_i$. We compute the signature of $Y$ by using decomposition of $f$. Let $J$ be the $4\times 4$ matrix:
$$J = 
\begin{bmatrix}
	0 & 1 & 0 & 0  \\
 -1  & 0 & 0 & 0 \\
  0 & 0 & 0 & 1  \\
	0 & 0 &-1 & 0  
\end{bmatrix}$$ and denote the product $(C_4C_3C_2C_1)$ by $\phi.$ The matrix representation of $\phi$ is
$$ \begin{bmatrix}
	0 & 1 & 0 & -1  \\
 -1  & 0 & 0 & 0 \\
  1 & 0 & 1 & 1  \\
	-1 & 0 &-1 & 0  
\end{bmatrix}.$$ 
Since each vanishing cycle is nonseparating and dual curve of each can be chosen disjoint from monodromy, by Lemma \ref{lemma:l4} and Theorem \ref{theorem:2}, the signature of $Y;$ $\sigma_{\phi}=0.$ Now consider the Lefschetz fibration with monodromy $\phi^2.$ Again by Wall's formula 
\[
\sigma_{\phi^2}=\sigma_{\phi}+\sigma_{\phi}-\sigma_{\phi^2_{c}}
\]
and by a result of Theorem \ref{theorem:t5}, 
\[
\sigma_{\phi^2_{c}}=\mathrm{sign}(\phi^{T}J-J\phi)=\mathrm{sign}\left (
\begin{bmatrix}
	2 & 0 & 1 & 1  \\
 0  & 2 & 0 & -1 \\
  1 & 0 & 2 & 1  \\
	1 & -1 &1 & 2  
\end{bmatrix} \right ).
\]
Thus $\sigma_{\phi^2}=-4.$ By squaring the $\phi^2$, we have a new Lefschetz fibration with monodromy $\phi^4$ and its signature is $$\sigma_{\phi^4}=\sigma_{\phi^2}+\sigma_{\phi^2}-\sigma_{\phi^4_{c}}$$ where 
\[
\sigma_{\phi^4_{c}}=\mathrm{sign}((\phi^2)^{T}J-J(\phi^2))=\mathrm{sign}\left (
\begin{bmatrix}
	0 & 1 & 1 & -1 \\
  1 & 0 & 2 & 1  \\
  1 & 2 & 2 & 0  \\
-1 & 1 & 0 & 0  
\end{bmatrix} \right ).
\] 
This gives $\sigma_{\phi^4}=-8.$ Now consider the Lefschetz fibration with mondromy $\phi^5$. At this step, decomposition is applied with $\phi^4$ and $\phi.$ By Wall's non-additivity formula 
\vspace{-0.029cm}
\[
\sigma_{\phi^5}=\sigma_{\phi^4}+\sigma_{\phi}-\sigma_{\phi^5_{c}}
\] 
where  
\begin{align*}
\sigma_{\phi^5_{c}}=\mathrm{sign} [(\phi^4)^{T}J+(\phi^3)^{T}J+(\phi^2)^{T}J+(\phi)^{T}J \\
&\hspace{-2.2cm}-J(\phi)-J(\phi^2)-J(\phi^3)-J(\phi^4)] .
\end{align*}
Since $\mathrm{sign}\left (
\begin{bmatrix}
	0 & 1 & 1 & -1 \\
  1 & 0 & 2 & 1  \\
  1 & 2 & 2 & 0  \\
-1 & 1 & 0 & 0  
\end{bmatrix}\right )=4,$ $\sigma_{\phi^5}=-12.$
By squaring the monodromy $\phi^5,$ one can easily have a Lefschetz fibration with monodromy $\phi^{10}$ whose signature is $-24.$ 
\end{example}


\begin{figure}[h]
	\includegraphics[width=0.37\textwidth]{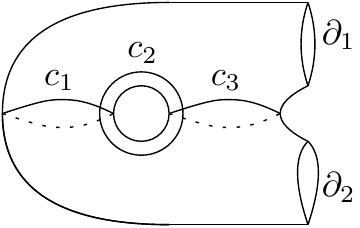}
	\label{fig:den}
	\caption{A surface with two boundary}
\end{figure}

\begin{example}
Consider the Lefschetz fibration with monodromy $$(C_3C_2C_1)^{4} =\delta_1\delta_2.$$  We will show that  the signatures of the Lefschetz fibrations corresponding to the left hand side and right hand side of this equation are $-7$ and $-1$ respectively.  Our computation is consistent with a result of Endo and Nagami \cite[Proposition 3.10]{en}  where it was proven  that the difference of the signatures of these two Lefschetz fibrations is $-6$. To compute $\sigma_{\delta_1\delta_2}$ and $\sigma_{(C_3C_2C_1)^{4}},$ attach a $1$-handle connecting the two boundary components and then cap off the boundary. By above discussion, the signature is unchanged. Note that $\delta_1=\delta_2=\delta$
in $\h_1(\widehat{\Sigma};\R)$ where $\widehat{\Sigma}$ is a surface below. 

\begin{figure}[h]
	\includegraphics[width=0.40\textwidth]{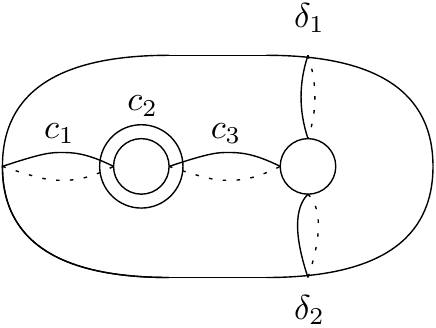}
	\label{fig:den}
	\caption{The surface $\widehat{\Sigma}$ }
\end{figure}

By the Wall non-additivity formula,
\[
\sigma_{\delta_1\delta_2}=2\sigma_{\delta}-\sigma_{\delta_{c}^2}
\]
where $\sigma_{\delta_{c}^2}=\mathrm{sign} (\delta_{*}^{T}J-J\delta_*)$ and $\delta_{*}=
 \begin{bmatrix}
	1 & 0 & 0 & 0  \\
  0 & 1 & 0 & 0  \\
  0 & 0 & 1 & 1  \\
	0 & 0 & 0 & 1 
\end{bmatrix}.$ Since $\delta$ is non-separating in  $\widehat{\Sigma}, \sigma_{\delta}=0$ and a direct computation shows $\sigma_{\delta_{c}^2}=1.$ Hence $\sigma_{\delta_1\delta_2}=-1.$ Now let $\phi=C_3C_2C_1.$ By Lemma~\ref{lemma:l4}, $\sigma_{\phi}=0.$

\begin{align*}
\sigma_{\phi^2}&=2\sigma_{\phi}-\mathrm{sign}(\phi^TJ-J\phi)=-3,\\
\sigma_{\phi^4}&=2\sigma_{\phi^2}-\mathrm{sign}((\phi^2)^TJ-J\phi^2)=-7.
\end{align*}
\end{example}


\section{Positive signature}
In this section, we will prove Theorem \ref{thm:positive}. First We need a simple obsevation about the signatures of a special class of $2\times 2$ matrices. Consider the following subset of $2 \times 2$ matrices:
\[
\mathscr{A}= \{A \in M_{2 \times 2} : A_{12}=A_{21}, \; A_{11} < 0, A_{22}> 0\}.
\]

Under the matrix addition $\mathscr{A}$ is a monoid. Note that since each $A \in \mathscr{A}$ is a symmetric matrix and diagonalizable with real eigenvalues, $ \mathrm{det A} < 0$. Therefore $\mathrm{sign (A)}=0$.

\begin{lemma} \label{lemma:l8}
Let $B$ be a matrix with positive entries. Then for $k\in\{1,\ldots, n\},$ 
\[
(B^{T})^kJ-JB^k \in \mathscr{A}.  
\]
Consequently, $\displaystyle \sum_{k=1}^{n} ((B^{T})^{k}J-JB^{k}) \in \mathscr{A}$. 
\end{lemma}

\begin{proof}
Let $B=$
$\begin{bmatrix}
	B_{11} & B_{12} \\
	B_{21} & B_{22}
\end{bmatrix}$
in $\mathscr{B}$. Then
$B^{T}J-JB=$
$\begin{bmatrix}
	-B_{21}-B_{21} & B_{11}-B_{22} \\
	-B_{22}+B_{11} & B_{12}+B_{12}
\end{bmatrix}$
which is in $\mathscr{A}$. Then for each $k$, $((B^{T})^{k}J-JB^{k}) \in \mathscr{A}$ also holds since power of positive matrices are again positive. As $\mathscr{A}$ is a monoid, the sum $\displaystyle \sum_{k=1}^{n} ((B^{T})^{k}J-JB^{k})$ is in $\mathscr{A}$ .
\end{proof}

We will also need the following example which is due to Ozbagci. 

\begin{lemma}(Ozbagci) There exists a Lefschetz fibration over $\mathbb{D}^2$ with fiber genus one and one boundary component, and the signature of the total space is $+1$.

\end{lemma}
\begin{proof} Let $\Sigma=\Sigma_1^1$ be a surface of genus one with one boundary component. As $H_1(\Sigma;\mathbb{Z})=\mathbb{Z}\times \mathbb{Z}$, we can denote its homology classses by row vectors $[a\;b]$ where $a$ and $b$ are integers. 
Let $\gamma_1, \gamma_2$ and $\gamma_3$ be simple closed curves representing the homology classes $\gamma_3=[1 \; 5], \gamma_2=[2 \; 5]$ and $\gamma_1=[1 \; 0]$ respectively. The homology action of Dehn twist along $\gamma_3$ is $C_3=$
$\begin{bmatrix}
	-4 & 1\\
	-25 & 6
\end{bmatrix}$, along $\gamma_2$ is $C_2=$
$\begin{bmatrix}
	-9 & 4 \\
	-25 & 11
\end{bmatrix}$, and along $\gamma_1$ is $C_1=$
$\begin{bmatrix}
	1 & 1\\
	0 & 1
\end{bmatrix}$.

Let $f:Y\rightarrow \D^2$ be a Lefschetz fibration with three vanishing cycles $\gamma_1,\gamma_2$ and $\gamma_3$ on the regular fiber $\Sigma_1^1$ and the monodromy factorization of $f$ is $\phi=C_3C_2C_1=$
$\begin{bmatrix}
	11 & 6\\
	75 & 41
\end{bmatrix}$.

By using the formula in Theorem \ref{theo:main}, it can be shown that the signature of the total space is $1$. 

\end{proof}
Now we will construct a Lefschetz fibration with signature $n$.

Let $\widetilde{f}:\widetilde{Y}\rightarrow \D^2$ be the  Lefschetz fibration with regular fiber $\Sigma_1^1$ given by the monodromy factorization $\phi^n=(C_3C_2C_1)^n$ and $\widetilde{Y}$ is the $n$-fold cyclic branched cover of $Y$ branched along a regular fiber. By the Theorem \ref{thm:bra}, we have
$$\sigma(\widetilde{Y})=n\sigma(Y)-\sum_{k=1}^{n-1} \sigma_{\phi^{k+1}_{c}}.$$
We will show that $\sum_{k=1}^{n-1} \sigma_{\phi^{k+1}_{c}}=0$.
For each $k$, $\sigma_{\phi^{k+1}_{c}}=0$ implies that  the signature of
 $\displaystyle \sum_{i=1}^{k} ((\phi^{T})^{i}J-J\phi^{i}) \in \mathscr{A}$ is zero by the Equation \eqref{eq:eq10}. It sufficies to show that the determinant of this $2\times 2$ matrices summation is negative.
By Lemma~\ref{lemma:l8}, the matrices in the equation belong to $ \mathscr{A}$. The result follows for $g=b=1$.

For $g> 1, b\geq 1$, we attach one handles as required to make fiber $\Sigma_g^b$ without changing the signature . Finally if $b=0$, we first do the construction for $b=1$ and cap off the boundary as in Theorem~\ref{thm:clos}.

\section*{Acknowledgements} We are grateful to Burak Ozbagci for sharing his old notes with us. A special thanks goes to Ferit \"Ozt\"urk for noticing a mistake in the earlier version.  While working on this project A\c{C} was supported by TUBITAK postdoctoral fellowship BIDEB-2218, No:1929\text{B}011700264 (2017/2) and \c{C}K was supported by BAGEP award of the Science Academy and  Bo\u{g}azi\c{c}i University Research Fund Grant Number 12482.

\bibliography{References}
\bibliographystyle{amsalpha}

\end{document}